\documentclass[12pt]{article}
\usepackage{etex}
\usepackage{amsmath,amsthm,amsfonts,amssymb,mathbbol,geometry,mathabx,amssymb}
\usepackage{cite}
\usepackage{amscd}
\usepackage{authblk}
\usepackage{tikz}
\usepackage{tikz-cd}
\usetikzlibrary{matrix,arrows,decorations.pathmorphing}
\usetikzlibrary{topaths}

\usepackage[export]{adjustbox}

\usepackage[unicode,
linktocpage,
colorlinks,
CJKbookmarks=true,
bookmarks=true
bookmarksnumbered=true,
pdfstartview=FitH,
linkcolor=orange,
citecolor=blue,
urlcolor=magenta]{hyperref}
\usepackage[all,cmtip]{xy}
\theoremstyle{plain}
\newtheorem{theorem}{Theorem}[section]
\newtheorem{proposition}[theorem]{Proposition}
\newtheorem{lemma}[theorem]{Lemma}

\theoremstyle{definition}
\newtheorem{definition}[theorem]{Definition}

\theoremstyle{remark}
\newtheorem{remark}[theorem]{Remark}

\numberwithin{equation}{section}

\newcommand{\prop}{\mathrm{prop}}
\newcommand{\Ind}{\mathrm{Ind}}
\newcommand{\sgn}{\mathrm{sgn}}
\newcommand{\ev}{\mathrm{ev}}
\newcommand{\supp}{\mathrm{Supp}}
\title{On localized signature and higher rho invariant of fibered manifolds}
\author[1]{Liu Hongzhi}
\author[2]{Wang Jinmin}
\affil[1]{School of Mathematics, Shanghai University of Finance and Economics, \protect\\ Shanghai 200433, P. R. China. \protect\\ e-mail: liu.hongzhi@mail.shufe.edu.cn}
 
\affil[2]{Shanghai Center for Mathematical Sciences, Fudan University,
\protect\\ Shanghai 200433, P. R. China. \protect\\ e-mail: wangjm15@fudan.edu.cn}

\date{}

\begin{document}
	
	\maketitle

	\abstract{
Higher index of signature operator is a far reaching generalization of signature of a closed oriented manifold.
When two closed oriented manifolds are homotopy equivalent, one can define a secondary invariant of the relative signature operator called higher rho invariant. The higher rho invariant detects the topological nonrigidity of a manifold. In this paper, we prove product formulas for higher index and higher rho invariant of signature operator on fibered manifolds. Our result implies the classical product formula for numerical signature of fiber manifolds obtained by Chern, Hirzebruch, and Serre in \cite{CSH57}. We also give a new proof of the product formula for higher rho invariant of signature operator on product manifolds, which is parallel to the product formula for higher rho invariant of Dirac operator on product manifolds obtained by Xie and Yu in \cite{XY14P} and  Zeidler in \cite{Z16}.

	}
	\section{Introduction}
	The signature of a $4k$-dimensional manifold is defined to be the signature of the cup product as a non-degenerate symmetric bilinear form on the vector space of $2k$-cohomology classes. In \cite{CSH57}, Chern, Hirzebruch and Serre established a product formula of signature for fibered manifold. More precisely, let $F \to E\to B$ be a fibered manifold with base manifold $B$ and fiber manifold $F$, if $\pi_1(B)$ acts trivially on $H^*_{dR}(F)$, the de Rham cohomology of $F$. We have the following product formula
	\begin{equation}\label{equa ancestor formula}
	\sgn(B)\times \sgn(F)=\sgn(E).
	\end{equation}

	The signature of a manifold is also equal to the Fredholm index of the signature operator.
	When taking into account of the fundamental group of the manifold, one can introduce higher invariants of the signature operator, which lie in the $K$-theory of certain geometric $C^*$-algebras.
	Let $M$ be an $m$-dimensional manifold with fundamental group $\pi_1(M)=G$ and universal covering $\widetilde{M}$. Let $D^{sgn}_M$ be the signature operator on $M$. The higher index of $D^{sgn}_M$, $\Ind(D^{sgn}_{M})$ is a generalization of the Fredholm index, and is defined to be an element in $K_m(C^*(\widetilde{M})^{G})$, where $C^*(\widetilde{M})^{G}$ is the equivariant Roe algebra of $\widetilde M$ and is Morita equivalent to the reduced group $C^*$-algebra $C^*_r(G)$.
	The higher index of signature operator is invariant under homotopy equivalence and oriented cobordism, and plays a fundamental role in the study of classification of manifolds. 
	On the other hand, $D^{sgn}_M$ defines a $K$-homology class $[D^{sgn}_M]$ in $K_m(C^*_L(\widetilde{M})^G)$, the $K$-theory of the equivariant localization algebra. See Section \ref{sec:preliminary} and \ref{section Signature class of Hilbert-Poincare complex  } for the explicit definitions of equivariant geometric $C^*$-algebras, higher index and $K$-homology class of signature operator.
	
	Furthermore, if $f: M'\to M$ is an orientation-preserving homotopy equivalence of closed manifolds, then there exists a concrete homotopy path that realizes the equality
	\[
	\text{Ind}(D^{sgn}_{M'})=\text{Ind}(D^{sgn}_{M})\in K_m(C^*(\widetilde{M})^G),
	\]
	where $m$ is the dimension of $M$ and $M'$.  
	This homotopy path allows one to define a secondary invariant of signature operator associated to the homotopy equivalence $f$, called higher rho invariant, in the $K$-theory of the equivariant  obstruction algebra $C^*_{L,0}(\widetilde{M})^G$.
	The higher rho invariant of signature operator associated to homotopy equivalence plays a central role in estimating the topological nonrigidity of a manifold (cf: \cite{HR3, PS16, Z17, WXY16, JL19}).
	
	Inspired by Chern, Hirzebruch and Serre's product formula, we prove a product formula for higher index and higher rho invariant of signature operator on fibered manifold. More precisely, consider a closed fibered manifold $F\to E\to B$ with base space $B$ and fiber $F$. Denote the fundamental group of $E$ by $G$, and the fundamental group of $B$ by $H$. Let $\widetilde E$ and $\widetilde{B}$ be the universal covering spaces of $E$ and $B$ respectively. Set $n=\dim F$ and $m=\dim B$. We first define equivariant family localization algebra $C^*_L(\widetilde{E};\widetilde{B})^G$, and family obstruction algebra $C^*_{L,0}(\widetilde{E};\widetilde{B})^G$. We show that there are naturally defined product maps:
\begin{eqnarray*}
	\phi : K_{m}( C_{L}^*(\widetilde{B})^H )  \otimes   K_{n}(C^*_{L}(\widetilde{E},\widetilde{B})^G )&\to& K_{m+n} (C^*_L(\widetilde{E})^G ),\\
\phi_0:  K_{m}( C_{L}^*(\widetilde{B})^H )  \otimes   K_{n}(C^*_{L,0}(\widetilde{E},\widetilde{B})^G )&\to& K_{m+n} (C^*_{L,0}(\widetilde{E})^G ).
\end{eqnarray*}
Taking advantage of the fiberwise signature operator, we introduce the family $K$-homology class of family signature operator along $F$, denoted by $[D^{sgn}_{E,B}]$, in the $K$-theory of the equivariant family localization algebra, and the family higher rho invariant $\rho(f;B)$, associated to a fiberwise homotopy equivalence $f:E'\to E$, in the $K$-theory of the equivariant family obstruction algebra. 
The following theorem is a product formula for $K$-homology class of signature operator on fibered manifold, which implies the product formula for higher index of signature operator.
\begin{theorem}\label{theo:main 1}
Let $F\to E\to B$ be fibered manifold with base space $B$ and fiber $F$. Denote the fundamental group of $E$ by $G$, and the fundamental group of $B$ by $H$. Let $\widetilde E$ and $\widetilde{B}$ be the universal covering spaces of $E$ and $B$ respectively. Write $\dim F=n$ and $\dim B=m$. Let $[D^{sgn}_{E,B}]$ be the family $K$-homology class of the family signature operator in $K_n(C^*_{L}(\widetilde{E},\widetilde{B})^G)$. We have the following product formula for family $K$-homology class of family signature operator
		\[
		k_{mn}\cdot\phi ([D_B^{sgn}] \otimes [D_{E,B}^{sgn}] )= [D_E^{sgn}],
		\]
where $k_{mn}=1$ when $mn$ is even and $k_{mn}=2$ otherwise, and $\phi$ is the product map
\[\phi : K_{m}( C_{L}^*(\widetilde{B})^H )  \otimes   K_{n}(C^*_{L}(\widetilde{E},\widetilde{B})^G )\to K_{m+n} (C^*_{L}(\widetilde{E})^G ). \]
\end{theorem}
We also obtain the following product formula for higher rho invariant of signature operator on fibered manifold.
\begin{theorem}\label{theo:main 2}Let $F\to E\to B$ and $F'\to E'\to B$ be two fibered manifolds with base space $B$ and fiber $F$ and $F'$ respectively. Let $f: E'\to E$ be a fiberwise homotopy equivalence. Denote the fundamental group of $E$ and $E'$ by $G$, and the fundamental group of $B$ by $H$. Let $\widetilde E, \widetilde{E'}$ and $\widetilde{B}$ be the universal covering spaces of $E$, $E'$ and $B$ respectively. Write $\dim F=n$ and $\dim B=m$. Let $\rho(f;B)$ be the family higher rho invariant associated to fiberwise homotopy equivalence $f$ defined in $K_n(C^*_{L,0}(\widetilde{E},\widetilde{B})^G)$. We have the following product formula for higher rho invariant associated to fiberwise homotopy equivalence
		\[
		k_{mn}\cdot\phi_0 ([D_B^{sgn}] \otimes \rho(f;B) )= \rho(f),
		\]
where $k_{mn}=1$ when $mn$ is even and $k_{mn}=2$ otherwise, and $\phi_0$ is the product map
\[\phi_0 : K_{m}( C_{L}^*(\widetilde{B})^H )  \otimes   K_{n}(C^*_{L,0}(\widetilde{E},\widetilde{B})^G )\to K_{m+n} (C^*_{L,0}(\widetilde{E})^G ). \]
\end{theorem}

 As an application of Theorem \ref{theo:main 1}, we give an alternative proof of the product formula of Chern, Hirzebruch and Serre (cf: \cite{CSH57}). Also, the product formula of higher rho invariant stated in Theorem \ref{theo:main 2} can be applied to study the behavior of higher rho map in \cite{WXY16} under fibration, and thus can be applied to study the topological nonrigidty of fibered manifold.

	We mention that the product formula for higher index of signature operator has been obtained by Wahl in \cite{Wahl10}. In this paper, we give a new proof of Wahl's product formula. On the other hand, product formula for higher rho invariant for positive scalar curvature metric on product manifolds has been proved by Siegel in his thesis \cite{Siegel12}, by Xie and Yu in \cite{XY14P}, and by Zeidler in \cite{Z16}. Their results and Theorem 6.8, 6.9 in \cite{WXY16} inspire us to study the product formula for higher rho invariant for signature operators.
	
The paper is organized as follows. In Section \ref{sec:preliminary}, we briefly recall some definitions of geometric $C^*$-algebras that we may use throughout the paper. In Section \ref{section Signature class of Hilbert-Poincare complex  }, we revisit the construction of several higher invariants associated to the signature operator. Next in Section \ref{sec:product formula}, we prove the product formulas for higher index and higher rho invariant of signature operator on product manifolds. In Section \ref{sec:product for fiber}, we generalize the product formulas to fibered manifold and prove Theorem \ref{theo:main 1} and \ref{theo:main 2}. We shall define an auxiliary $C^*$-algebra consisting of operators that can be localized along the base manifold, and use the Mayer-Vietoris arguments.

	\section{Preliminary}\label{sec:preliminary}
	The aim of this section is to briefly recall some basic definitions of geometric $C^*$-algebras used throughout the paper.
	For more details, we refer the readers to \cite{WXY16}.

 Let $X$ be a proper metric space and $G$ be a finitely represented discrete group. Suppose that $G$ acts on $X$ properly by isometries.	For simplicity, we assume that the $G$-action is free. Let $C_0(X)$ be the $C^*$-algebra consisting of all complex-valued continuous functions on $X$ that vanish at infinity.
 An $X$-module is a separable Hilbert space $H_X$ equipped with a $*$-representation of $C_0(X)$. It is called nondegenerate if the $*$-representation is nondegenerate, and standard if no nonzero function in $C_0(X)$ acts as a compact operator. Additionally we assume that $H_X$ is equipped with a unitary representation of $G$ which is compatible with $C_0(X)$-representation, that is,
$$\forall f\in C_0(X),\ g\in G,\ \pi(g) \phi(f)=\phi(g.f)\pi(g)$$
 where $\phi$ (resp. $\pi$) is the $C_0(X)$(resp. $G$)-representation on $H_X$ and $g.f(x)= f(g^{-1}x)$.



	Now let us recall the definitions of propagation of operator and locally compact operator. 
	\begin{definition}\label{def propa local pseudo}
		Under the above assumptions, let $T$ be a bounded linear operator acting on $H_X$.
		\begin{enumerate}
			\item The \emph{propagation} of $T$ is defined by
			\begin{equation}
			\prop(T)=\sup\{d(x,y) ~|~ (x,y)\in \text{Supp}(T) \},
			\end{equation}
			 where $\text{Supp}(T)$ is the complement (in $X\times X$) of the set of points $(x,y) \in X\times X$ such that there exists $f_1,f_2\in C_0(X)$ such that $f_1Tf_2=0$ and $f_1(x)f_2(y)\neq 0$;
			\item $T$ is said to be \emph{locally compact} if both $fT$ and $Tf$ are compact for all $f\in C_0(X)$.
		\end{enumerate}
	\end{definition}
	
	In the following, we recall the definitions of the equivariant Roe algebra, localization algebra, and the obstruction algebra.
	\begin{definition}\label{def roe and localization}
		Let $H_X$ be a standard nondegenerate $X$-module and $B(H_X)$ the set of all bounded linear operators on $H_X$.
		\begin{enumerate}
			\item The $G$-equivariant $Roe$ algebra of $X$, denoted by $C^*(X)^G$, is the $C^*$-algebra generated by all $G$-equivariant locally compact operators with finite propagation in $B(H_X)$.
			\item The $G$-equivariant localization algebra $C_L^*(X)^G$ is the $C^*$-algebra generated by all uniformly norm-bounded and uniformly norm-continuous functions $f : [1,\infty)\to C^*(X)^G$ such that 	
			\[
			\prop(f(t))<\infty \text{ and }\prop(f(t))\to 0\  \text{~as~}  t \to \infty.
			\]
			\item The $G$-equivariant obstruction algebra $C^*_{L,0}(X)^G$ is defined to be the kernel of the following evaluation map
			\begin{eqnarray*}
				\ev:\;C_L^*(X)^G&\to & C^*(X)^G\\
				f&\mapsto  & f(1)
			\end{eqnarray*}
			In particular, $C^*_{L,0}(X)^G$ is an ideal of $C_L^*(X)^G$.
			
			
		\end{enumerate}
	\end{definition}

	
	%

	\begin{remark}
		Up to isomorphism, $C^*(X)^G$ does not depend on the choice of the standard nondegenerate $X$-module.
		The same holds for $C_L^*(X)^G$ and $C_{L,0}^*(X)^G$. 
	\end{remark}
	
	When $X$ is a Galois $G$-covering of a closed Riemannian manifold,  $L^2(X)$ is a standard nondegenerate $X$-module. In this case, there is an equivalent definition of equivariant Roe algebra.
	\begin{definition}\label{def:precompletion} Let $X$ be a Galois $G$-covering of a closed Riemannian manifold.
		Set $\mathbb{C}[X]^G$ as a $*$-algebra consisting of integral operators given by
		$$\forall f\in L^2(X),\ f\mapsto \int_Xk(x,y)f(y)dy,$$
		where $k: X\times X\to \mathbb{C} $ is uniformly continuous, bounded on $X\times X$, and has finite propagation, i.e.
		$$\exists \delta>0,\text{ s.t. }k(x,y)=0, \text{\ if \ }d(x,y)>\delta,$$
		and is $G$-equivariant, i.e. $k(gx,gy)=k(x,y)$ for any $g\in G$. The $G$-equivariant equivariant Roe algebra is the operator norm completion of $\mathbb{C}[X]^G$.
	\end{definition}

\begin{remark}
	If we remove the cocompactness of the $G$-action, Definition \ref{def:precompletion} actually gives us the uniform equivariant Roe algebra, which is different from the equivariant Roe algebra.
\end{remark}

	Suppose that $T\in \mathbb{C}[X]^G$ has corresponding Schwartz kernel $k(x,y)$. The support of $T$ defined in Definition \ref{def roe and localization} is simply given by
	\begin{equation}\label{eq:support}
	\supp(T)=\overline{\{(x,y)\in X\times X:k(x,y)\ne 0\}}.
	\end{equation}
	
	Similarly we define the $G$-equivariant localization algebra $C^*_L(X)^G$ to be the completion of all paths on
	$t\in [0,+\infty)$ with value in $\mathbb{C}[X]^G$, which are uniformly continuous and uniformly bounded with respect to the operator norm, and have propagation going to zero as $t$ goes to infinity.
	
 Definition \ref{def:precompletion} can be easily generalized to the case where $L^2(X)$ is replaced by $L^2$-section of an Hermitian vector bundle over $X$ on which $G$ acts by isometries. The above definitions coincide with Definition \ref{def roe and localization}, as those $C^*$-algebras are independent of the choice of the standard nondegenerate $X$-module.
	
	Now we consider a product of two manifolds.
	Let $M, N$ be two closed manifolds and $\widetilde{M}$, $\widetilde{N}$ be their Galois $G$,$H$-covering spaces respectively. For any pair of integral operators in $\mathbb{C}[\widetilde M]^G$ and $\mathbb{C}[\widetilde N]^H$, their tensor product is well-defined as an operator in $\mathbb{C}[\widetilde M\times \widetilde N]^{G\times H}$. This induces the following product maps
	\begin{equation}\label{eq:K-product}
	\begin{split}
			\psi: K_m(C^*(\widetilde{M})^G) \otimes K_n(C^*(\widetilde{N})^H) &\to K_{m+n} (C^*_L(\widetilde{M}\times \widetilde{N})^{G\times H}),\\
			\psi_L: K_m(C^*_L(\widetilde{M})^G) \otimes K_n(C^*_L(\widetilde{N})^H) &\to K_{m+n} (C^*_L(\widetilde{M}\times \widetilde{N})^{G\times H}),\\
			\psi_{L,0}: K_m(C^*_L(\widetilde{M})^G) \otimes K_n(C^*_{L,0}(\widetilde{N})^H) &\to K_{m+n} (C^*_{L,0}(\widetilde{M}\times \widetilde{N})^{G\times H}).
	\end{split}
	\end{equation}
	
%
%
	
	\section{Higher invariant associated to signature operator}\label{section Signature class of Hilbert-Poincare complex  }
	
	In this section, we recall a formula of the higher index of signature operator, which was obtained by Higson and Roe in \cite{HR1} and \cite{HR2}. After that we will give the construction of $K$-homology class of the signature operator, which was originally introduced by Weinberger, Xie and Yu in \cite{WXY16}. At last, we give the definition of the higher rho invariant of controlled homotopy equivalence of manifolds.

    In this section, all manifolds mentioned are not assumed to be compact or connected unless otherwise noted.
    \subsection{Higher index of signature operator}\label{sus:higher sig}
     Let $M$ be a Riemannian manifold of dimension $m$. Let $\widetilde{M}$ be a Galois $G$-covering space of $M$, where $G$ is a finitely represented discrete group. Denote by $\Lambda^p(\widetilde{M})$ the $L^2$-completion of compactly support smooth differential $p$-forms on $\widetilde{M}$. Let $d_{\widetilde{M}}$ be the usual differential operators, which is an unbounded operator from $\Lambda^{p}(\widetilde{M})$ to $\Lambda^{p+1}(\widetilde{M})$. We will write $\Lambda^{even}(\widetilde{M})=\oplus_{k}\Lambda^{2k}(\widetilde{M})$,
	$\Lambda^{odd}(\widetilde{M})=\oplus_{k}\Lambda^{2k+1}(\widetilde{M})$,
	and $\Lambda(\widetilde M)=\Lambda^{even}(\widetilde{M})\oplus \Lambda^{odd}(\widetilde{M})$.
	
	Set $D_{\widetilde{M}}=d_{\widetilde{M}}+ d_{\widetilde{M}}^*$. Let $*$ be the Hodge $*$-operator of $\Lambda^*(\widetilde{M})$. Define
$S_{\widetilde{M}}$ as
	\begin{eqnarray*}
		S_{\widetilde{M}}: \Lambda^p(\widetilde{M})&\to & \Lambda^{n-p}(\widetilde{M}),\\
		\omega & \to & i^{p(p-1)+[\frac{n}{2}]}  *\omega.
\end{eqnarray*}

We now recall the representative of higher index of signature operator according to the parity of $m$.

 \paragraph{Odd case.}
  When $m$ is odd, the signature operator $D^{sgn}_M$ is given by
 \[
 iD_{\widetilde{M}}S_{\widetilde{M}}: \Lambda^{even}(\widetilde{M})\to \Lambda^{even}(\widetilde{M}).
 \]
 It is shown in \cite{HR1} and \cite{HR2} that

 \begin{enumerate}
  \item $D_{\widetilde{M}}\pm S_{\widetilde{M}}$ are both invertible;
  \item the invertible operator
  \[(D_{\widetilde M}+ S_{\widetilde M})(D_{\widetilde{M}}- S_{\widetilde{M}})^{-1}: \Lambda^{even}(\widetilde{M})\to \Lambda^{even}(\widetilde{M})
   \]
    belongs to $(C^*(\widetilde{M})^{G})^+$, thus defines a class in $K_1(C^*(\widetilde{M})^{G})$ denoted by $\Ind(D^{sgn}_{M})$;
  \item the higher index of the signature operator is equal to $\Ind(D^{sgn}_{M})$.
  \end{enumerate}

\paragraph{Even case.}
The even case is parallel.  When $m$ is even, the signature operator $D_M^{sgn}$ is $D_{\widetilde{M}}$ together with the grading operator $S_{\widetilde{M}}$.

It is shown in \cite{HR1} and \cite{HR2} that
\begin{enumerate}
  \item $D_{\widetilde{M}}\pm S_{\widetilde{M}}$ are both invertible;
  \item Let $P_{+}(D_{\widetilde{M}}\pm S_{\widetilde{M}}) $ be the positive projections of invertible operators $D_{\widetilde{M}}\pm S_{\widetilde{M}}$ respectively. Then $P_{+}(D_{\widetilde{M}}\pm S_{\widetilde{M}}) $ can be approximated by operators with finite propagation, and
  \[
  P_{+}(D_{\widetilde{M}}+ S_{\widetilde{M}}) - P_{+}(D_{\widetilde{M}}- S_{\widetilde{M}}) \in C^*(\widetilde{M})^G,
  \]
  thus the formal difference
  \[
  [P_{+}(D_{\widetilde{M}}+ S_{\widetilde{M}})] - [P_{+}(D_{\widetilde{M}}- S_{\widetilde{M}}) ]
  \]
  determines a $K$-theory class in $K_0(C^*(\widetilde{M})^{G})$ denoted by $\Ind(D^{sgn}_{M})$;
  \item the higher index of the signature operator is equal to $\Ind(D^{sgn}_{M})$.
  \end{enumerate}

  \subsection{$K$-homology class of signature operator}
	
   In this subsection we recall the definition of $K$-homology class of signature operator according to the parity of $\dim M$.

\paragraph{Odd case.}	Suppose that $M$ is odd dimensional. For any $n\in \mathbb{N}^+$, let $M_{n}$ be the manifold $M$ equipped with metric $g_{n}$, which is $n$-times of the original metric $g$.
Denote by $\coprod_n M_{n}$ the disjoint union of $M_n$.
	The higher index of signature operator $D^{sgn}_{\coprod M_{n}}$ is represented by the invertible operator
	$$T=(D_{\coprod_n \widetilde M_{n}}+S_{\coprod_n \widetilde M_{n}})(D_{\coprod_n \widetilde M_{n}}-S_{\coprod_n \widetilde M_{n}})^{-1}\Big|_{\Lambda^{even}(\coprod_n \widetilde M_{n})}.$$
	
	Denote by $T_n$ the restriction of $T$ to $M_n$. As shown in the previous subsection, $T$ can be approximated by operators with finite propagation on $\coprod_n M_{n}$. Therefore with respect to the original metric $g$, $\{T_n\}$ can be approximated by a sequence of operators uniformly whose propagation goes to zero as $n$ goes to infinity.
	
	To make the above sequence a continuous path on $[1,+\infty)$, we further consider $\coprod_n M_{n+r}$ for $r\in[0,1]$ rather than $\coprod_n M_{n}$, where the metric on $M_{n+r}$ is $(n+r)$-times of the original metric. This actually gives an invertible element in $C^*_L(\widetilde M)^G$.
	\begin{definition}\label{def:K-hom sig odd}
		We call the $K$-theory element in $K_1(C^*_L(\widetilde{M})^G)$ represented by the invertible element defined above the $K$-homology class of signature operator, which will be denoted by $[D^{sgn}_M]$.
	\end{definition}

\paragraph{Even case.}We sketch the construction of $K$-homology class of signature operator for the case of $m$ being even but leave out the details.
 The higher index of signature operator $D^{sgn}_{\coprod M_{n}}$ is determined by the difference of projections
	$$P_+(D_{\coprod \widetilde M_{n}}+S_{\coprod \widetilde M_{n}})-P_+(D_{\coprod \widetilde M_{n}}-S_{\coprod \widetilde M_{n}}).$$
Again, we consider the higher index of signature operator on $\coprod_n M_{n+r}$ for $r\in[0,1]$. This gives us a continuous path from $[1,+\infty)$ to $C^*(\widetilde M)^G$, which in turn defines a path of difference of projections lies in $C^*_L(\widetilde M)^G$, denoted by $P_{M,+}(t)-P_{M,-}(t)$, with respect to the original metric $g$.
\begin{definition}\label{def:K-hom sig even}
		We call the $K$-theory element determined by the formal difference $[P_{M,+}(t)]-[P_{M,-}(t)]$ the $K$-homology class of signature operator $D^{sgn}_M$, which will be denoted by $[D^{sgn}_M]\in K_0(C^*_L(\widetilde{M})^G)$.
	\end{definition}

	\subsection{Controlled homotopy equivalence and higher rho invariant}

    In this subsection, we recall the construction of higher rho invariant of signature operator associated to homotopy equivalence.

    Higher rho invariant associated to smooth homotopy equivalence was first introduced by Higson and Roe in \cite{HR1,HR2,HR3}. Later in \cite{PS16}, Piazza and Schick gave an index theoretic definition of higher rho invariant of signature operator. In \cite{Z17}, Zenobi extended Higson and Roe, Piazza and Schick's work to define higher rho invariant associated to topological homotopy equivalence.

	In \cite{WXY16}, Weinberger, Xie and Yu constructed higher rho invariant of signature operator associated to homotopy equivalence by piecewise-linear approach. In this paper, we adapt their construction to give a differential geometric approach to the definition of higher rho invariant. It is not hard to see that our construction here is equivalent to the one given in \cite{XY19}.  

\begin{definition}\label{def controlled homotopy equivalencce}
	Let $M'$ and $M$ be two Riemannian manifolds. Let $f: M'\to M$ be a smooth homotopy equivalence with smooth homotopy inverse $g:M\to M'$. Denote by $h_t', t\in [0,1]$ (resp. $h_t,t\in [0,1]$) the smooth homotopy between $fg$ and $\text{id}_{M'}$ (resp. $gf$ and $\text{id}_M$). We say that $f$ is a controlled homotopy equivalence if there exists a positive constant $C$ such that
\begin{enumerate}
\item the diameter of $\{h'_t(a)| 0\leq t\leq 1 \}$ is bounded by $C$ uniformly for all $a\in M'$,
\item the diameter of $\{h_t(b)| 0\leq t\leq 1  \}$ is bounded by $C$ uniformly for all $b\in M$.
\end{enumerate}
\end{definition}
\begin{remark}
	If $M'$ and $M$ are closed manifold, then any homotopy equivalence $f:M'\to M$ is automatically controlled. Furthermore, the lift of $f$ to their Galois covering is also controlled.
\end{remark}

	Let $f: M'\to M$ be a controlled homotopy equivalence.
	Suppose that $\widetilde{M'}, \widetilde{M}$ are Galois $G$-covering spaces of $M,M'$ respectively. Write
	$D=D_{\widetilde{M'}}\oplus D_{\widetilde{M}}$, and $S=
	\left(
	\begin{array}{cc}
	S_{\widetilde{M'}} &  \\
	& -S_{\widetilde{M}}
	\end{array}
	\right)$ acting on $\Lambda^*(\widetilde M')\otimes \Lambda^*(\widetilde M)$.
	The controlled homotopy equivalence $f: M'\to M$ induces a map from $\Lambda^p(\widetilde{M})$ to
	$\Lambda^p(\widetilde{M'})$, which we will still denote by $f$.
	In general, the induced map is not a bounded operator. However,  we may apply the Hilsum-Skandalis submersion (cf:\cite[Page74]{HS92}, \cite[Page 157]{W13}, and \cite[Page 34]{XY19}) to construct a bounded operator $\mathcal{T}_f$ out of $f$. Without loss of generality, we might as well assume that $f$ is a bounded operator.
	
Now let us recall the definition of higher rho invariant associated to controlled homotopy equivalence $f$ according to the parity of $\dim M$. We mention here that the construction is due to Higson and Roe (cf: \cite{HR1}, \cite{HR3}), and Weinberger, Xie and Yu (cf: \cite{WXY16}).

\paragraph{Odd case} We first assume that both $M'$ and $M$ are odd dimensional.
	Via conjugating by $f$ on the first summand $\Lambda^*(\widetilde M')$, we may identify $D$ and $S$ with their corresponding operators acting on $\Lambda^*(\widetilde M)\oplus \Lambda^*(\widetilde M)$.
	Under this identification, the invertible element defined by
	$$(D+S)(D-S)^{-1}\Big|_{\Lambda^{even}(\widetilde{M'})\oplus \Lambda^{even}(\widetilde{M})}\in M_2(C^*_L(\widetilde M)^G)^+$$
	represents the $K$-theory element $f_*\Ind(D^{sgn}_{M'})-\Ind(D^{sgn}_{M})\in K_1(C^*(\widetilde M)^G)$.
	The construction in Definition \ref{def:K-hom sig odd} gives rise to an invertible element
	$$(D_t+S_t)(D_t-S_t)^{-1}\Big|_{\Lambda^{even}(\widetilde{M'})\oplus \Lambda^{even}(\widetilde{M})}\in M_2(C^*_L(\widetilde M)^{G,+}).$$
	In particular, we have $D_1=D$ and $S_1=S$.
	Since $f$ is a controlled homotopy equivalence, it gives rise to a canonical path that connects $(D+S)(D-S)^{-1}$ to the identity operator as shown by Higson and Roe in\cite{HR1}. The path is constructed out of the following path $S_f(t)$ connecting $S$ with $-S$,	\begin{equation}\label{eq:Sf(t)}S_f(t)=
	\begin{cases}
	\begin{pmatrix}
	(1-t)S_{\widetilde{M'}} + t f^* S_{\widetilde{M}} f & 0\\
		0 &  -S_{\widetilde{M}}
	\end{pmatrix}&t\in[0,1]\\
	\begin{pmatrix}
		\cos (\frac{\pi}{2}(t-1)) f^* S_{\widetilde{M}} f & \sin (\frac{\pi}{2}(t-1)) f^* S_{\widetilde{M}}  \\
		\sin (\frac{\pi}{2}(t-1))   S_{\widetilde{M}} f   & - \cos (\frac{\pi}{2}(t-1)) S_{\widetilde{M}}
	\end{pmatrix}&t\in [1,2]\\
	\begin{pmatrix}
	0 &  e^{i\pi (t-2)}S_{\widetilde{M}} f\\
	e^{-i\pi (t-2)}f^*S_{\widetilde{M}}   & 0
	\end{pmatrix}&t\in[2,3]\\
	-\begin{pmatrix}
		0 &  e^{i\pi (4-t)}S_{\widetilde{M}} f\\
		e^{-i\pi (4-t)}f^*S_{\widetilde{M}}   & 0
	\end{pmatrix}&t\in[3,4]\\
	-\begin{pmatrix}
		\cos (\frac{\pi}{2}(5-t)) f^* S_{\widetilde{M}} f & \sin (\frac{\pi}{2}(5-t)) f^* S_{\widetilde{M}}  \\
		\sin (\frac{\pi}{2}(5-t))   S_{\widetilde{M}} f   & - \cos (\frac{\pi}{2}(5-t)) S_{\widetilde{M}}	
	\end{pmatrix}&t\in[4,5]\\
	-\begin{pmatrix}
		(t-5)S_{\widetilde{M'}} + (6-t) f^* S_{\widetilde{M}} f & 0\\
		0 &  -S_{\widetilde{M}}
	\end{pmatrix}&t\in[5,6].
	\end{cases}
	\end{equation}
%
	For any $t\in[0,6]$, $S_f(t)$ satisfying the following conditions:
\begin{enumerate}
  \item $D\pm S_{f}(t)$ are both invertible;
  \item the invertible operator
  \[\frac{D+ S}{D- S_f(t)}: \Lambda^{even}(\widetilde{M})\oplus \Lambda^{even}(\widetilde{M}')\to \Lambda^{even}(\widetilde{M})\oplus \Lambda^{even}(\widetilde{M}')\]
   belongs to $M_2(C^*(\widetilde{M})^{G})^+$.
  \end{enumerate}
	\begin{definition}\label{def:higher rho odd}
		The higher rho invariant $\rho(f)$ is the $K$-theory class in $K_1(C^*_{L,0}(\widetilde M)^G)$ represented by the following invertible element
		\begin{equation}
		\begin{cases}
		(D+S)(D+S_f(t-1))^{-1}\Big|_{\Lambda^{even}(\widetilde{M'})\oplus \Lambda^{even}(\widetilde{M})},
		& t\in[1,7],\\
		(D_{t-6}+S_{t-6})(D_{t-6}-S_{t-6})^{-1}\Big|_{\Lambda^{even}(\widetilde{M'})\oplus \Lambda^{even}(\widetilde{M})},
		& t\geqslant 7.
		\end{cases}
		\end{equation}
	\end{definition}

\paragraph{Even case.}
The even dimensional case is parallel to the odd case above.
The construction in Definition \ref{def:K-hom sig even} gives rise to a path of difference of projections
	$$P_+(D_t+S_t)-P_+(D_t-S_t)\in M_2(C^*_L(\widetilde M)^{G}),$$ with
$$P_+(D_1+S_1)-P_+(D_1-S_1)=P_+(D+S)-P_+(D-S).$$
Let $S_f(t)$ be as above. Similarly, we have that:
\begin{enumerate}
  \item $D\pm S_{f}(t)$ are both invertible;
  \item $P_{+}(D \pm S_f(t)) $ are of finite propagation, and
  \[
  P_{+}(D+ S) - P_{+}(D- S_f(t)) \in M_2( C^*(\widetilde{M})^{G}).
  \]
    \end{enumerate}
  Thus the formal difference
  \[
  [P_{+}(D+ S)] - [P_{+}(D- S_f(t)) ]
  \]
  determines a $K$-theory class in $K_0(C^*(\widetilde{M})^{G})$;

  The higher rho invariant associated to controlled homotopy equivalence $f$ is defined as follows.
  \begin{definition}\label{def:higher rho even}
  	Write
  	\begin{equation}
  	\begin{split}
  	\Theta_{f, +}(t)=\begin{cases}
  	P_+(D+S)
  	& t\in[1,7]\\
  	P_+(D_{t-6}+S_{t-6})
  	& t\geqslant 7
  	\end{cases}\\
  	\Theta_{f,-}(t)=\begin{cases}
  	P_+(D+S_f(t-1))
  	& t\in[1,7]\\
  	P_+(D_{t-6}-S_{t-6})
  	& t\geqslant 7
  	\end{cases}
  	\end{split}
  	\end{equation}
As $\Theta_{f, \pm}(t)$ are projections and their difference lies in $M_2(C^*_{L,0}(\widetilde{M})^{G})$, the formal difference
\[
[\Theta_{f, +}]-[\Theta_{f,-}]
\]
defines a $K$-theory class $\rho(f)$ in $K_0(C^*_{L,0}(\widetilde M)^G)$, called higher rho invariant.
	\end{definition}


	\section{Product formula}\label{sec:product formula}
	In this section, we will prove the product formula for higher rho invariant associated to the signature operator for homotopy equivalence. We only consider the case for product of manifolds for now. The general case for fibered manifolds will be discussed in the next section.

\begin{proposition}\label{prop pro for of sig of pro mfld}
Let $M,\ N$ be two manifolds with dimension $m, \ n$ and fundamental groups $G,H$ respectively. Under the product map
\[
\psi: K_m(C^*(\widetilde{M})^G)\otimes K_n(C^*(\widetilde{N})^{H})\to K_{m+n} (\widetilde{M}\times \widetilde{N})^{G\times H} ),
\]
we have
\[
k_{mn}\cdot\psi(\Ind(D^{sgn}_{M})\otimes \Ind(D^{sgn}_{N})) =\Ind(D^{sgn}_{M\times N}),
\]
where
\[
k_{mn}=
\begin{cases}
1, & mn\ \text{is even},\\
2, & mn\ \text{is odd}.
\end{cases}
\]
\end{proposition}
\begin{proof}
In the following, we omit the mention of $\psi$ for simplicity. We avoid to use the fact that $S_{\widetilde{M}}^2=1$ throughout the proof for the purpose of further generalization. Therefore, we have to consider four cases according to the parity of both $\dim M$ and $\dim N$.

	\paragraph{Even times odd.}
We first suppose that $\dim M$ is even and $\dim N$ is odd.

Write $B_{\widetilde M\pm}=D_{\widetilde M}\pm S_{\widetilde M}$ for short. On the product manifold $\widetilde M\times\widetilde N$, the differential operator $d_{\widetilde M\times\widetilde N}$ is given by
$$d_{\widetilde M\times\widetilde N}=
d_{\widetilde M}\hat\otimes 1+1\hat\otimes d_{\widetilde N}
=d_{\widetilde M}\otimes 1+E_{\widetilde M}\otimes d_{\widetilde N},$$
where $E_{\widetilde M}$ is the even-odd grading operator for $\Lambda(\widetilde M)$.
Therefore
$$D_{\widetilde M\times\widetilde N}=D_{\widetilde M}\otimes 1+E_{\widetilde M}\otimes D_{\widetilde N}.$$

Now we decompose $\Lambda^{even}(\widetilde M\times\widetilde N)$ into the direct sum of
$\Lambda^{even}(\widetilde M)\otimes\Lambda^{even}(\widetilde N)$ and
$\Lambda^{odd}(\widetilde M)\otimes\Lambda^{odd}(\widetilde N)$. As $\dim N$ is odd, the Hodge $*$-operator as well as the Poincar\'e duality operator $S_{\widetilde N}$ reverses the parity of $\Lambda(\widetilde N)$. Note that $S_{\widetilde N}^2:\Lambda^{p}(\widetilde N)\to\Lambda^{p}(\widetilde N)$ is a multiple of identity. Therefore we identify
$\Lambda(\widetilde M)\otimes\Lambda^{odd}(\widetilde N)$ as
$\Lambda(\widetilde M)\otimes\Lambda^{even}(\widetilde N)$ via $1\otimes S_{\widetilde N}$. Under this identification, the higher index of signature operator is represented by the following invertible operator
\begin{eqnarray*}
	(B_{\widetilde M}^+\otimes 1+1\otimes S_{\widetilde N}D_{\widetilde N})(B_{\widetilde M}^-\otimes 1+1\otimes S_{\widetilde N}D_{\widetilde N})^{-1}:\\
	\Lambda(\widetilde M)\otimes\Lambda^{even}(\widetilde N)\to
	\Lambda(\widetilde M)\otimes\Lambda^{even}(\widetilde N).
\end{eqnarray*}

Since $B_{\widetilde M}^+$ is invertible, we define a path of bounded operators
$$W_{+,s}=\frac{B_{\widetilde M}^+}{|B_{\widetilde M}^+|^s}\otimes 1+1\otimes S_{\widetilde N}D_{\widetilde N}.$$

For the invertible operator $B_{\widetilde M}^\pm$, we denote by $P^+(B_{\widetilde M}^\pm)$ (resp. $P^-(B_{\widetilde M}^\pm)$) the spectral projection of the positive (resp. negative) part of $B_{\widetilde M}^\pm$. We see that
$$W_{+,0}=B_{\widetilde M}^+\otimes 1+1\otimes S_{\widetilde N}D_{\widetilde N},\ W_{+,1}=(P^+(B_{\widetilde M}^+)-P^-(B_{\widetilde M}^+))\otimes 1+1\otimes S_{\widetilde N}D_{\widetilde N}.$$
Since $D_{\widetilde N}$ anti-commutes with $S_{\widetilde N}$, we have
$$W_{+,s}^*W_{+,s}=(B_{\widetilde M}^+)^{2(1-s)}\otimes 1+1\otimes D_{\widetilde N}^2>0.$$
Thus $W_{+,s}$ is a path of invertible operator for every $s$ in $[0,1]$.

Similarly we define a path of invertible operator $W_{-,s}$. Thus via the homotopy $W_{+,s}(W_{-,s})^{-1}$, the higher index of signature operator on $M\times N$ is also represented by the invertible operator
$$W_{+,1}(W_{-,1})^{-1}:
\Lambda(\widetilde M)\otimes\Lambda^{even}(\widetilde N)\to
\Lambda(\widetilde M)\otimes\Lambda^{even}(\widetilde N).$$
 We rewrite the expression above using $1=P^+(B_{\widetilde M}^\pm)+P^-(B_{\widetilde M}^\pm)$.
\begin{align*}
&[W_{+,1}(W_{-,1})^{-1}]\\
=&[P^+(B_{\widetilde M}^+)\otimes(S_{\widetilde N}D_{\widetilde N}+1)(S_{\widetilde N}D_{\widetilde N}-1)^{-1}+P^-(B_{\widetilde M}^+)\otimes 1]\\
&-[P^+(B_{\widetilde M}^-)\otimes(S_{\widetilde N}D_{\widetilde N}+1)(S_{\widetilde N}D_{\widetilde N}-1)^{-1}+P^-(B_{\widetilde M}^-)\otimes 1]\\
=&([P^+(B_{\widetilde M}^+)]-[P^+(B_{\widetilde M}^-)])\times
[(D_{\widetilde N}+S_{\widetilde N})(D_{\widetilde N}-S_{\widetilde N})^{-1}]\\
=&\Ind(D^{sgn}_M)\times \Ind(D^{sgn}_N).
\end{align*}
The last two equalities follow from the definition of product of $K$-groups and the formula of the higher index of signature operator in Section \ref{sus:higher sig}.
\paragraph{Odd times even.} Suppose that $M$ is odd dimensional and $N$ is even dimensional. Straightforward computation shows that
    \begin{enumerate}
    \item $S_{\widetilde{M}\times \widetilde{N}} = S_{\widetilde{M}} \otimes S_{\widetilde{N}} $ on $\Lambda(\widetilde{M})\otimes \Lambda^{even}(\widetilde{N}) $,
    \item $S_{\widetilde{M}\times \widetilde{N}} = -S_{\widetilde{M}} \otimes S_{\widetilde{N}} $ on $\Lambda(\widetilde{M})\otimes \Lambda^{odd}(\widetilde{N}) $.
    \end{enumerate}
    and
    \begin{enumerate}
    \item $d_{\widetilde{M}\times \widetilde{N}} = d_{\widetilde{M}} \otimes 1 + 1\otimes d_{\widetilde{N}} $ on $\Lambda(\widetilde{M})^{odd}\otimes \Lambda(\widetilde{N}) $,
    \item $d_{\widetilde{M}\times \widetilde{N}} = d_{\widetilde{M}} \otimes 1 -1 \otimes  d_{\widetilde{N}} $ on $\Lambda(\widetilde{M})^{even}\otimes \Lambda(\widetilde{N}) $.
    \end{enumerate}
  Let $\Lambda_{\pm}(\widetilde{N})$ be the $\pm 1$ eigenspace of $S_{\widetilde{N}}$. We make the following identifications
  \begin{enumerate}
  \item Under the decomposition
  $$
  \Lambda^{odd}(\widetilde{M})\otimes \Lambda(\widetilde{N})= \begin{matrix}
  \Lambda^{odd}(\widetilde{M})\otimes \Lambda_+^{odd}(\widetilde{N})\\
  \oplus\\
    \Lambda^{odd}(\widetilde{M})\otimes \Lambda_-^{odd}(\widetilde{N})\\
  \oplus\\
    \Lambda^{odd}(\widetilde{M})\otimes \Lambda_+^{even}(\widetilde{N})\\
  \oplus\\
    \Lambda^{odd}(\widetilde{M})\otimes \Lambda_-^{even}(\widetilde{N})
  \end{matrix}
  \text{ and }
   \Lambda^{odd}(\widetilde{M}\times \widetilde{N})= \begin{matrix}
  \Lambda^{even}(\widetilde{M})\otimes \Lambda_+^{odd}(\widetilde{N})\\
  \oplus\\
    \Lambda^{even}(\widetilde{M})\otimes \Lambda_-^{odd}(\widetilde{N})\\
  \oplus\\
    \Lambda^{odd}(\widetilde{M})\otimes \Lambda_+^{even}(\widetilde{N})\\
  \oplus\\
    \Lambda^{odd}(\widetilde{M})\otimes \Lambda_-^{even}(\widetilde{N})
    \end{matrix},
  $$
we identify $  \Lambda^{odd}(\widetilde{M})\otimes \Lambda(\widetilde{N})$ with $ \Lambda^{odd}(\widetilde{M}\times \widetilde{N})$ by
\[
\begin{pmatrix}
- B_{\widetilde{M}+} \otimes 1 & & & \\
& B_{\widetilde{M}-}\otimes 1 & & \\
& &  1\otimes 1  & \\
& & & 1\otimes 1
\end{pmatrix}:
\Lambda^{odd}(\widetilde{M})\otimes \Lambda(\widetilde{N})\to \Lambda^{odd}(\widetilde{M}\times \widetilde{N}).
\]

\item Under the decomposition
 $$ \Lambda^{even}(\widetilde{M}\times N)= \begin{matrix}
  \Lambda^{even}(\widetilde{M})\otimes \Lambda_+^{even}(\widetilde{N})\\
  \oplus\\
    \Lambda^{even}(\widetilde{M})\otimes \Lambda_-^{even}(\widetilde{N})\\
  \oplus\\
    \Lambda^{odd}(\widetilde{M})\otimes \Lambda_+^{odd}(\widetilde{N})\\
  \oplus\\
    \Lambda^{odd}(\widetilde{M})\otimes \Lambda_-^{odd}(\widetilde{N})
    \end{matrix}
  \text{ and }
  \Lambda^{even}(\widetilde{M})\otimes \Lambda(N)= \begin{matrix}
  \Lambda^{even}(\widetilde{M})\otimes \Lambda_+^{even}(\widetilde{N})\\
  \oplus\\
    \Lambda^{even}(\widetilde{M})\otimes \Lambda_-^{even}(\widetilde{N})\\
  \oplus\\
    \Lambda^{even}(\widetilde{M})\otimes \Lambda_+^{odd}(\widetilde{N})\\
  \oplus\\
    \Lambda^{even}(\widetilde{M})\otimes \Lambda_-^{odd}(\widetilde{N})
  \end{matrix},$$
  we identify $  \Lambda^{even}(\widetilde{M}\times \widetilde{N})$ with $\Lambda^{even}(\widetilde{M})\otimes \Lambda(\widetilde{N})$ by
  \[
  \begin{pmatrix}
  1\otimes 1 & & & \\
  & -1\otimes 1 & & \\
  & &  - B_{\widetilde{M}-}\otimes 1 & \\
  & & &  - B_{\widetilde{M}+}\otimes 1
  \end{pmatrix}: \Lambda^{even}(\widetilde{M}\times \widetilde{N})\to  \Lambda^{even}(\widetilde{M})\otimes \Lambda(\widetilde{N}).
  \]
  \end{enumerate}

With these identifications, we have
\[
d_{\widetilde{M}\times \widetilde{N}} + d^*_{\widetilde{M}\times \widetilde{N}}+S_{\widetilde{M}\times \widetilde{N}}=\left
\{\begin{array}{cc}
B_{\widetilde{M}+} \otimes 1 + B_{\widetilde{M}+} \otimes D_{\widetilde{N}} & \text{on}\ \Lambda^{odd}(\widetilde{M})\otimes \Lambda^{even}_+(\widetilde{N})\\
B^2_{\widetilde{M}-} B_{\widetilde{M}+}\otimes 1 + B_{\widetilde{M}+} \otimes D_{\widetilde{N}} & \text{on}\ \Lambda^{odd}(\widetilde{M})\otimes \Lambda^{odd}_+(\widetilde{N})\\
-B_{\widetilde{M}-} \otimes 1 + B_{\widetilde{M}-} \otimes D_{\widetilde{N}} & \text{on}\ \Lambda^{odd}(\widetilde{M})\otimes \Lambda^{even}_-(\widetilde{N})\\
-B^2_{\widetilde{M}+} B_{\widetilde{M}-}\otimes 1 + B_{\widetilde{M}-} \otimes D_{\widetilde{N}} & \text{on}\ \Lambda^{odd}(\widetilde{M})\otimes \Lambda^{odd}_+(\widetilde{N})
\end{array}
\right.
\]
Note that $B^2_{\widetilde{M}\pm}$ are positive invertible operators. It follows that $d_{\widetilde{M}\times \widetilde{N}} + d^*_{\widetilde{M}\times \widetilde{N}}+S_{\widetilde{M}\times \widetilde{N}}$ is homotopic to
\[
\begin{pmatrix}
B_{\widetilde{M}+} & 0 \\
 0  & B_{\widetilde{M}-}
\end{pmatrix}S_{\widetilde{N}} + \begin{pmatrix}
B_{\widetilde{M}-} & 0\\
0   & B_{\widetilde{M}+}
\end{pmatrix}D_{\widetilde{N}} : \Lambda^{odd}(\widetilde{M})\otimes  \Lambda(\widetilde{N}) \to \Lambda^{even}(\widetilde{M})\otimes  \Lambda(\widetilde{N}),
\]
where the matrix form is written with respect to the decomposition
\[
 \Lambda(\widetilde{N}) =\Lambda_+ (\widetilde{N}) \oplus \Lambda_- (\widetilde{N}).
\]
Since $D_{\widetilde{N}}$ is off-diagonal,  $d_{\widetilde{M}\times \widetilde{N}} + d^*_{\widetilde{M}\times \widetilde{N}}+S_{\widetilde{M}\times \widetilde{N}}$ is in turn homotopic to
\[
V= \begin{pmatrix}
B_{\widetilde{M}+} & 0\\
0   & B_{\widetilde{M}-}
\end{pmatrix}S_{\widetilde{N}}f(D_{\widetilde{N}}) + \begin{pmatrix}
B_{\widetilde{M}-} & 0 \\
0   & B_{\widetilde{M}+}
\end{pmatrix}g(D_{\widetilde{N}})
\]
where $g(x)=\frac{x}{\sqrt{1+x^2}}$ and $f(x)=\frac{1}{\sqrt{1+x^2}}$. In the meantime, $d_{\widetilde{M}\times \widetilde{N}} + d^*_{\widetilde{M}\times \widetilde{N}}-S_{\widetilde{M}\times \widetilde{N}}$ is homotopic to
\[
U= \begin{pmatrix}
B_{\widetilde{M}-} & 0 \\
0   & B_{\widetilde{M}+}
\end{pmatrix}S_{\widetilde{N}}f(D_{\widetilde{N}}) + \begin{pmatrix}
B_{\widetilde{M}-} & 0 \\
0   & B_{\widetilde{M}+}
\end{pmatrix}g(D_{\widetilde{N}}).
\]
It follows that
\[VU^{-1}: \Lambda^{even}(\widetilde{M}\times \widetilde{N}) \to \Lambda^{even}(\widetilde{M}\times \widetilde{N})\]
\[
VU^{-1} = [(d_{\widetilde{M}}+d^*_{\widetilde{M}})\otimes 1 + S_{\widetilde{M}} \otimes S_2S_1S_2 ] \begin{pmatrix}B^{-1}_{\widetilde{M}-} & \\
& B^{-1}_{\widetilde{M}+}\end{pmatrix},
\]
where
\[
S_1= S_{\widetilde{N}}, \ \text{and} \ S_2= g(D_{\widetilde{N}})+ S_{\widetilde{N}}f(D_{\widetilde{N}}).
\]

Note that $ S_2S_1S_2 $ is a symmetry, i.e. $S_2S_1S_2 $ can be approximated by finite propagation operators, and $(S_2S_1S_2)^2-1$ belongs to $C^*(\widetilde{N})^H$. We define
\[
P=\frac{S_2S_1S_2+1}{2}.
\]

Now one can see that the higher index of  $D^{sgn}_{M\times N}$ is actually represented by
\begin{eqnarray*}
&&[(d_{\widetilde{M}}+d^*_{\widetilde{M}})\otimes 1 + S_{\widetilde{M}} \otimes S_2S_1S_2 ] \begin{pmatrix}B^{-1}_{\widetilde{M}-} &0 \\
0& B^{-1}_{\widetilde{M}+}\end{pmatrix}\\
&=& [B_{\widetilde{M}+}\otimes P + B_{\widetilde{M}-} \otimes (1-P) ]( \begin{pmatrix}B^{-1}_{\widetilde{M}-} & 0  \\
0&0 \end{pmatrix}+ \begin{pmatrix}0 &0  \\
0 & B^{-1}_{\widetilde{M}+}\end{pmatrix})\\
&= & \begin{pmatrix} B_{\widetilde{M}+}B^{-1}_{\widetilde{M}-} & 0  \\
0 & 1 \end{pmatrix} \otimes P  + \begin{pmatrix} 1 & 0 \\
0 &B_{\widetilde{M}-}  B^{-1}_{\widetilde{M}+}\end{pmatrix}\otimes (1-P)\\
&=& ( \begin{pmatrix} B_{\widetilde{M}+}B^{-1}_{\widetilde{M}-} &  0\\
0 & B_{\widetilde{M}+}B^{-1}_{\widetilde{M}-}  \end{pmatrix} \otimes P  + \begin{pmatrix} 1 & 0 \\
0 &1\end{pmatrix}\otimes (1-P)) \begin{pmatrix} 1 & 0  \\
0 &B_{\widetilde{M}-}  B^{-1}_{\widetilde{M}+}\end{pmatrix}\\
&=& B_{\widetilde{M}+}B^{-1}_{\widetilde{M}-} \otimes ([P]-
[\begin{pmatrix}
0&0 \\
0& 1
\end{pmatrix}]),
\end{eqnarray*}
which is equal to $\Ind(D^{sgn}_{M})\otimes \Ind(D^{sgn}_{N})$
since
\[
[P]-
[\begin{pmatrix}
0&0 \\
0& 1
\end{pmatrix}]=[P_+(D_{\widetilde{N}}+S_{\widetilde{N}})]-[P_+(D_{\widetilde{N}}-S_{\widetilde{N}})],
\]
as computed in \cite{HR2}.
	\paragraph{Even times even.}
 Suppose that both $M$ and $N$ are even dimensional. In \cite{WXY16}, it is shown that
 \begin{align*}
 \Ind(D^{sgn}_{N})\otimes \Ind(D^{sgn}_{\mathbb R})=& \Ind(D^{sgn}_{N\times \mathbb R}),\\
 \text{Ind}(D_{M\times N}^{sgn})\otimes \text{Ind}(D_{\mathbb{R}}^{sgn})=&\text{Ind}(D_{M\times N\times \mathbb{R}}^{sgn}).
 \end{align*}
Note that $\Ind(D^{sgn}_{\mathbb R})$ is the generator of
\[
 K_1(C^*(\mathbb{R}))\cong K_1(C^*_L(\mathbb{R})) \cong \mathbb{Z}.
\]
Now we have
\begin{align*}
&\text{Ind}(D_M^{sgn}))\otimes \text{Ind}(D_N^{sgn}) \otimes \text{Ind}(D_{\mathbb{R}}^{sgn})\\
=& \text{Ind}(D_M^{sgn}))\otimes\text{Ind}(D_{N\times \mathbb{R}}^{sgn})\\
=& \text{Ind}(D_{M\times N\times \mathbb{R}}^{sgn})\\
=& \text{Ind}(D_{M\times N}^{sgn})\otimes \text{Ind}(D_{\mathbb{R}}^{sgn}).
\end{align*}
It follows that
\[
\text{Ind}(D_M^{sgn})\otimes \text{Ind}(D_N^{sgn})= \text{Ind}(D_{M\times N}^{sgn}).
\]
	\paragraph{Odd times odd.}
Let $M$ and $N$ be both odd dimensional manifolds. In this case, as shown in \cite{WXY16}, we have
\begin{eqnarray*}
2\Ind(D^{sgn}_{N})\otimes \Ind(D^{sgn}_{\mathbb R})&=& \Ind(D^{sgn}_{N\times\mathbb R}),\\
\Ind(D^{sgn}_{M\times N})\otimes \Ind(D^{sgn}_{\mathbb R})&=& \Ind(D^{sgn}_{M\times N\times\mathbb R}).
\end{eqnarray*}
Now we have
\begin{align*}
&2 \text{Ind}(D_M^{sgn}))\otimes \text{Ind}(D_N^{sgn})\otimes \text{Ind}(D_{\mathbb{R}}^{sgn})\\
=& \text{Ind}(D_M^{sgn})\otimes \text{Ind}(D_{N\times \mathbb{R}}^{sgn})\\
=& \text{Ind}(D_{M\times N\times \mathbb{R}}^{sgn})\\
=& \text{Ind}(D_{M\times N}^{sgn})\otimes \text{Ind}(D_{\mathbb{R}}^{sgn}).
\end{align*}
It follows that
\begin{eqnarray*}
2\text{Ind}(D_M^{sgn})\otimes \text{Ind}(D_N^{sgn})= \text{Ind}(D_{M\times N}^{sgn}).
\end{eqnarray*}

\end{proof}
Note that in the proof of Proposition of \ref{prop pro for of sig of pro mfld}, we actually do not demand that $S_{\widetilde M}^2=1$.
It follows that
given the definition of $K$-homology class of signature and higher rho invariant,  the argument above can be easily generalized to show the following Proposition and Theorem:

\begin{proposition}\label{prop product localized sig}
With the same notations, under the product map
		\[
		\psi_L: K_m(C^*_L(\widetilde{M})^G) \otimes K_n(C^*_L(\widetilde{N})^H) \to  K_{m+n} (C^*_L(\widetilde{M}\times \widetilde{N})^{G\times H}),
		\]
		there is a product formula of $K$-homology class of signature operator on $M$ and $N$ goes as follows,
		\[k_{mn}\cdot\psi_{L}([D_M^{sgn}]\otimes [D_N^{sgn}])= [D_{M\times N}^{sgn}]\]
where
\[
k_{mn}=\left\{
\begin{array}{cc}
1, & mn\ \text{is even},\\
2, & mn\ \text{is odd}.
\end{array}
\right.
\]
	\end{proposition}

\begin{theorem}\label{theo product of higher rho}
		Suppose that $M',M$ are two closed oriented Riemannian manifolds and $f:M'\to M$ is a homotopy equivalence. Write $m=\dim M'=\dim M$. Let $\widetilde{M'},\ \widetilde{M}$ be theirs Galois $G$-covering spaces respectively. Under the  product map
		\[	\psi_{L,0}: K_m(C^*_L(\widetilde{M})^G) \otimes K_n(C^*_{L,0}(\widetilde{N})^H) \to  K_{m+n} (C^*_{L,0}(\widetilde{M}\times \widetilde{N})^{G\times H}),\]
		there is a product formula goes as follows,
		\[
		k_{mn}\cdot\psi_{L,0} (\rho(f) \otimes [D_N^{sgn}]  )= \rho(f\times I_N),
		\]
		where $I_N: N\to N$ is the identity map, and
\[
k_{mn}=\left\{
\begin{array}{cc}
1, & mn\ \text{is even},\\
2, & mn\ \text{is odd}.
\end{array}
\right.
\]
	\end{theorem}

	\section{Product formula for fibered manifolds}\label{sec:product for fiber}
	In this section, we generalize the product formula given in the previous section to fibered manifolds. We will first introduce a series of family geometric $C^*$-algebras with respect to the fibration. Next we define the family version of $K$-homology class and higher rho invariant of fiberwise signature operator in $K$-theory of these $C^*$-algebras. At last, We will prove Theorem \ref{theo:main 1} and \ref{theo:main 2}.
	\subsection{Family algebras}
	In this subsection, we introduce family geometric $C^*$-algebras associated to a fibered manifold.
	
Let $\pi:E\to B$ be a fibration with fiber $F$ and base space $B$. Assume that $E$, $F$ and $B$ are closed connected oriented Riemannian manifolds. The fibration induces a long exact sequence of homotopy groups
	$$\xymatrix{\cdots\ar[r]&\pi_2(B)\ar[r]^\partial&\pi_1(F)\ar[r]&\pi_1(E)\ar[r]^{\pi_*}&\pi_1(B)\ar[r]&0}
	$$
	Denote by $\widetilde E$ and $\widetilde B$ the universal covering of $E$ and $B$. From the exactness of the above sequence, we see that $\partial(\pi_2(B))$ is a normal subgroup of $\pi_1(F)$. Write $\Gamma=\pi_1(F)/\partial(\pi_2(B))$. The above exact sequence shows that $\widetilde E$ is also a fibration on $\widetilde B$ with fiber projection $\widetilde \pi:\widetilde E\to\widetilde B$ and fiber $\widetilde F$, the Galois $\Gamma$-covering of $F$.
%
	
	From now on, we will write $G=\pi_1(E)$ and $H=\pi_1(B)$ for short. Recall that the equivariant Roe algebras $C^*(\widetilde E)^G$ is defined to be the completion of $G$-equivariant, locally compact operators with finite propagation as in Definition \ref{def roe and localization}. Now let us define the equivariant family Roe algebra.
	
	First we construct an equivariant Roe algebra bundle over $B$. View the fiber bundle $E$ over $B$ as gluing many pieces of local trivialization by a series of diffeomorphisms of $F$. More precisely, there exists an open cover $\{V_\alpha\}$ of $B$ and continuous maps $\varphi_{\alpha\beta}:V_\alpha\cap V_\beta\to \mathrm{Diff}(F)$ such that the fiber bundle $E$ is equivalent to the tuple $(V_\alpha\times F,\varphi_{\alpha\beta})$, that is, every continuous section $s$ of $E$ is equivalent to a series of continuous maps $s_\alpha:V_\alpha\to F$ satisfying the cocycle condition, $\varphi_{\alpha\beta}(x)s_\alpha(x)=s_\beta(x)$ for any $x\in V_\alpha\cap V_\beta$.
	
	By our previous arguments, when turning to the universal covering, $\widetilde E$ is also an $\widetilde F$-bundle over $\widetilde B$. Assume that every open set in $\{V_\alpha\}$ is small enough so that it lifts to a $\pi_1(B)$-equivariant open cover $\{U_j\}$ of $\widetilde B$, each open set of which is homeomorphic to the Euclidean space and trivialize $\widetilde E$. Also, the transition map lifts to $\{\psi_{ij}\}$ with $\psi_{ij}:U_i\cap U_j\to \mathrm{Diff}(\widetilde F)$. Recall the subspace $\mathbb{C}[\widetilde F]^{\Gamma}$ of the equivariant Roe algebra $C^*(\widetilde F)^{\Gamma}$ as in Definition \ref{def:precompletion}. For any $x\in U_i\cap U_j$, $\psi_{ij}(x)$ induces an isomorphism $\psi_{ij,*}(x)$ of $C^*(\widetilde F)^{\Gamma}$ by conjugation, which maps $\mathbb{C}[\widetilde F]^{\Gamma}$ to itself. This induces the following fiber bundle.
	
	\begin{definition}[Equivariant family Roe algebra]\label{def:family roe algebra}
	Recall that $G=\pi_1(E)$ and  $\Gamma=\pi_1(F)/\partial (\pi_2(B))$.	A continuous section of the fiber bundle given by $(\{U_i\},\{\psi_{ij,*}\})$ is defined by a series of norm-continuous maps $s_i:U_i\to C^*(\widetilde F)^{\Gamma}$ satisfying the cocycle condition, $\psi_{ij,*}(x)s_i(x)=s_j(x)$ for any $x\in U_i\cap U_j$.
		Let $\mathbb{C}[\widetilde E,\widetilde B]^G$ be the collection of uniformly norm-bounded and uniformly norm-continuous sections that are invariant under $\pi_1(B)$-action, and have uniformly finite propagation on $\widetilde B$. The norm of such a section $\{s_j\}$ is defined to be $\sup_{j}\sup_{x\in U_j}\|s_j(x)\|$. Denote the completion of $\mathbb{C}[\widetilde E,\widetilde B]^G$ by $C^*(\widetilde E,\widetilde B)^G$.
	\end{definition}

We mentioned that Definition \ref{def:family roe algebra} is related to the ``Groupoid Roe algebra" given by Tang, Willett and Yao ( \cite[Definition 3.6]{TWY16}).
%
	
	
	
	It is easy to verify that the above definition is independent of the local trivialization. Similarly, we define the corresponding equivariant family localization algebra.
	
	\begin{definition}[Equivariant family localization and obstruction algebras]	
		The equivariant family localization algebra $C^*_L(\widetilde{E},\widetilde{B})^G$ is the completion of uniformly norm-bounded and uniformly norm-continuous paths $s:[0,+\infty)\to C^*(\widetilde{E},\widetilde{B})^G$  such that the propagation of $s(t)$ goes to zero as $t$ goes to infinity uniformly on $\widetilde B$,
		where the norm of $s(t)$ is defined to be $\sup_{t\in [0,+\infty)}||s(t)||$.
		The equivariant family obstruction algebra  $C^*_{L,0}(\widetilde{E},\widetilde{B})^G$ is then defined to be the kernel of the family assembly map:
		\begin{eqnarray*}
			\text{ev}: C^*_L(\widetilde{E},\widetilde{B})^G &\to& C^*(\widetilde{E},\widetilde{B})^G \\
			s &\mapsto & s(1).
		\end{eqnarray*}
	\end{definition}
%
	
	\subsection{Product map of $K$-theory}\label{subsec: product map}
	In this subsection, we construct the productive map on the family $C^*$-algebras.
	
	\begin{theorem}\label{thm: product maps}
	Recall that $G=\pi_1(E)$ and $H=\pi_1(B)$. There are product maps
		\begin{eqnarray*}
			\phi : K_m(C^*_L(\widetilde{B})^H)\times K_n(C^*_L(\widetilde{E},\widetilde{B})^G) &\to & K_{m+n}(C^*_L(\widetilde{E})^G)\\
			\phi_0 : K_m(C^*_L(\widetilde{B})^H)\times K_n(C^*_{L,0}(\widetilde{E},\widetilde{B})^G) &\to & K_{m+n}(C^*_{L,0}(\widetilde{E})^G)
		\end{eqnarray*}
	that generalize the maps defined in \eqref{eq:K-product}.
	\end{theorem}
	
	\begin{proof}
		Without loss of generality, we assume that both $m$ and $n$ are even.
		We will only give in details the construction of
		\[
		\phi_0 : K_0(C^*_L(\widetilde{B})^H)\times K_0(C^*_{L,0}(\widetilde{E},\widetilde{B})^G) \to  K_{0}(C^*_{L,0}(\widetilde{E})^G).
		\]
		Suppose that $f_t\in (C^*_L(\widetilde{B})^H)^+$ represents a $K_0$-class, which has finite propagation that goes to zero as $t$ goes to infinity, and is a $1/10$-projection, that is, $f_t^*=f_t$ and $\|f_t^2-f_t\|<1/10$. Similarly, we suppose that $g_t\in (C^*_{L,0}(\widetilde E,\widetilde B)^G)^+$ is a $1/10$-projection, has finite propagation that goes to zero uniformly as $t$ goes to infinity, and satisfies that $g_1=1$. Furthermore, we assume that $f_t-1$ and $g_t-1$ are given by kernel operators acting on $L^2$-sections as in Definition \ref{def:precompletion}.
		
		Choose $r>0$ small enough such that for any $x\in B$, the restriction of the fiber bundle $E$ to the $r$-ball near $x$ is trivial. Since $f_t$ and $f_{t+M}$ are homotopic for any $M>0$, we may assume that the propagation of $f_t$ is smaller than $r$. By the local triviality, we define
		$$\phi([f_t]\otimes [g_t])=[(f_t-1)\otimes(g_t-1)+1],$$
		where $(f_t-1)\otimes(g_t-1)+1\in (C^*_{L,0}(\widetilde E)^G)^+$ is given by
		 		\begin{equation}\label{eq:product r}
		 		\begin{split}
		 		&\Big( \big((f_t-1)\otimes (g_t-1)\big)h\Big) (x,y) \\
		 		=&\int_{\widetilde B}\int_{\widetilde F} (f_t-1)(x,x')\otimes
		 		(g_t-1)_{x'}(y,y')h(x',y')dy'dx',
		 		\end{split}
		 		\end{equation}		
		 		with $h\in L^2(\widetilde E)$. The above expression makes sense as the propagation of $f_t$ is small enough. It is easy to verify that $(f_t-1)\otimes(g_t-1)+1$ is at most a $3/10$-projection, which gives rise to a $K_0$-class.
		
		Now passing to the matrix algebra and the Grothendieck group, we obtain the product map.
	\end{proof}
	
	\subsection{Family higher invariants}
	In this subsection, we introduce the family version of higher invariants of the signature operator on fibered manifold, and prove Theorem \ref{theo:main 1} and  \ref{theo:main 2}.
	
	On the fibered manifold $E$, the vertical differentials and Poincar\'{e} duality are well-defined as they are compatible with the transition maps. Thus the family $K$-homology class of the vertical signature operator
	$[D_{E,B}^{sgn}]\in K_{\dim F}(C_L^*(\widetilde{E},\widetilde{B})^G)$ is defined similarly as in Definition \ref{def:K-hom sig odd} and \ref{def:K-hom sig even}.
	\begin{theorem}\label{thm: fiber loc}
	We have the following product formula for family $K$-homology class of family signature operator along $F$ holds:
		\[
		k_{B,F}\cdot\phi ([D_B^{sgn}] \otimes [D_{E,B}^{sgn}] )= [D_E^{sgn}],
		\]
where $k_{B,F}=1$ when $\dim B\cdot\dim F$ is even, and $k_{B,F}=2$ otherwise, and $\phi$ is the product map
\[\phi : K_{\dim B}( C_{L}^*(\widetilde{B})^H )  \otimes   K_{\dim F}(C^*_{L}(\widetilde{E},\widetilde{B})^G )\to K_{\dim E} (C^*_L(\widetilde{E})^G ).\]
	\end{theorem}
	
	We shall also define the family higher rho invariant of a fiberwise homotopy equivalence. Suppose we have two fibrations over the same base
		$$\xymatrix{F' \ar[r]^{} &E'\ar[r]^{\pi'} &B}\text{ and }\xymatrix{F \ar[r]^{} &E\ar[r]^{\pi} &B}.$$
	Let $f:E'\to E$ be a fiberwise homotopy equivalence, that is, the following diagram commutes
	$$\xymatrix{E'\ar[rr]^{f}\ar[dr]_{\pi'}&& E\ar[dl]^{\pi} \\ &B&}$$
	as well as replacing $f$ with its homotopy inverse and the corresponding homotopy.
	 Using the vertical differential and the Poincar\'e duality operator, we define a family higher rho invariant $\rho(f;B)\in K_{\dim F}(C_{L,0}^*(\widetilde{E},\widetilde{B})^G)$ as in Definition \ref{def:higher rho odd}.

	\begin{theorem}\label{thm: fiber rho}
			With the same notation as above,
			we have the following product formula for family higher rho invariant associated to fiberwise homotopy equivalence holds:
			\[
			k_{B,F}\phi ([D_B^{sgn}] \otimes   \rho(f;B) )= \rho(f),
			\]
where $k_{B,F}=1$ when $\dim B\cdot\dim F$ is even, and $k_{B,F}=2$ otherwise, and $\phi_0$ is the product map
\[\phi_0 : K_{\dim B}(C_{L}^*(\widetilde{B})^H_r  ) \otimes   K_{\dim F}( C^*_{L,0}(\widetilde{E},\widetilde{B})^G) \to K_{\dim E}( C^*_{L,0}(\widetilde{E})^G ).\]
	\end{theorem}
	
	In the following, we only prove Theorem \ref{thm: fiber rho} in details. The proof for Theorem \ref{thm: fiber loc} is similar.
	
	We need some definitions to prepare for the proof of Theorem \ref{thm: fiber rho}.
	\begin{definition}\label{def:prop along base}
		For any element $T\in C^*(\widetilde{E})^G$, we define the propagation of $T$ along the base space $\widetilde B$ by
		$$\prop_{\widetilde B}(T)=\sup\{d(\widetilde\pi(x),\widetilde\pi(y)) : (x,y)\in \text{Supp}(T) \},$$
		where $\widetilde\pi$ is the lift of the fiber projection $\pi:E\to B$.
	\end{definition}
	We need the following $C^*$-algebra generated by elements in $C^*_{L,0}(\widetilde E)^G$ that can be localized horizontally. This is a generalization of the equivariant localization algebra defined in Definition \ref{def roe and localization}.
	\begin{definition}\label{def:localization at base}
		Define $C^*_{\widetilde B,L,0}(\widetilde E)^G$ to be the $C^*$-algebra generated by paths $f:[1,+\infty)\to C^*_{L,0}(\widetilde E)^G$ such that $f(s)$ is uniformly norm-continuous and uniformly norm-bounded, and its propagation along $\widetilde B$ is finite and goes to zero uniformly as $s$ goes to $\infty$. The norm of $f\in C^*_{\widetilde B,L,0}(\widetilde E)^G$ is given by the supreme of its norm in $C^*_{L,0}(\widetilde E)^G$, that is,
		$\|f\|=\sup_{s\geqslant 1}\|f(s)\|.$
	\end{definition}
	There is an evaluation map
	$$\ev:C^*_{\widetilde B,L,0}(\widetilde E)^G\to C^*_{L,0}(\widetilde E)^G,$$
	which induces a $K$-theoretical map denoted by $\ev_*$.
	
	If $X$ is a closed Riemannian manifold, the equivariant localization algebra $C^*_L(\widetilde X)^{\pi_1X}$ admits a Mayer-Vietoris sequence for a partition of $X$.
	More precisely, if $U_1,U_2$ are two open sets on $X$ and $\widetilde{U_1},\widetilde{U_2}$ are their lifts to $\widetilde X$, then we have the following six-term exact sequence
	(cf: \cite[Proposition 3.11]{Yulocalization}). 
	$$\xymatrix{K_0(C^*_L(\widetilde{U_1}\bigcap\widetilde{U_2})^{\pi_1X})\ar[r]&
		{\begin{matrix}
		K_0(C^*_L(\widetilde{U_1})^{\pi_1X}) \\
		\oplus \\
		K_0(C^*_L(\widetilde{U_2})^{\pi_1X})
		\end{matrix}}
		\ar[r]& K_0(C^*_L(\widetilde{U_1}\bigcup\widetilde{U_2})^{\pi_1X})\ar[dd]\\ &&\\
		K_1(C^*_L(\widetilde{U_1}\bigcup\widetilde{U_2})^{\pi_1X})\ar[uu]&
	{\begin{matrix}
			K_1(C^*_L(\widetilde{U_1})^{\pi_1X})\\ \oplus \\K_1(C^*_L(\widetilde{U_2})^{\pi_1X})
		\end{matrix}}
		\ar[l]&K_1(C^*_L(\widetilde{U_1}\bigcap\widetilde{U_2})^{\pi_1X})\ar[l]
	}$$
	As the $C^*$-algebra $C^*_{\widetilde B,L,0}(\widetilde E)^G$ is generated by elements that can be localized along $B$, it also admits a Mayer-Vietoris sequence as above for two open sets on the base space.
	\begin{proposition}\label{prop:MV}
		Let $U_1,U_2$ be two open sets on $B$ and $\widetilde{U_1},\widetilde{U_2}$ their lifts to $\widetilde B$. Let $\widetilde E_{\widetilde{U_1}}$ and $\widetilde E_{\widetilde{U_2}}$ be the restriction of $\widetilde E$ to $\widetilde{U_1}$ and $\widetilde{U_2}$ respectively. We have the following six-term exact sequence:
			$$\xymatrix{K_0(C^*_{\widetilde B,L,0}(\widetilde E_{\widetilde{U_1}}\bigcap\widetilde E_{\widetilde{U_2}})^G)\ar[r]&
				{\begin{matrix}
					K_0(C^*_{\widetilde B,L,0}(\widetilde E_{\widetilde{U_1}})^G)\\
					\oplus \\
					K_0(C^*_{\widetilde B,L,0}(\widetilde E_{\widetilde{U_2}})^G)
					\end{matrix}}
				\ar[r]& K_0(C^*_{\widetilde B,L,0}(\widetilde E_{\widetilde{U_1}}\bigcup\widetilde E_{\widetilde{U_2}})^G)\ar[dd]\\ &&\\
			K_1(C^*_{\widetilde B,L,0}(\widetilde E_{\widetilde{U_1}}\bigcup\widetilde E_{\widetilde{U_2}})^G)\ar[uu]&
			{\begin{matrix}
				K_1(C^*_{\widetilde B,L,0}(\widetilde E_{\widetilde{U_1}})^G)\\
				\oplus \\
				K_1(C^*_{\widetilde B,L,0}(\widetilde E_{\widetilde{U_2}})^G)
				\end{matrix}}
			\ar[l]&K_1(C^*_{\widetilde B,L,0}(\widetilde E_{\widetilde{U_1}}\bigcap\widetilde E_{\widetilde{U_2}})^G)\ar[l]
		}.$$
	\end{proposition}

\begin{proof}
We sketch the proof of Proposition \ref{prop:MV} as follows, which is essentially the same as the proof of Proposition 3.11 in \cite{Yulocalization}.
	For any open subset $Y$ of $B$, we define $(C^*_{\widetilde B,L,0}(\widetilde E)^G)_Y$ to be the $C^*$-subalgebra of $C^*_{\widetilde B,L,0}(\widetilde E)^G$ generated by all paths $f:[1,\infty)\to C^*_{L,0}(\widetilde E)^G$ such that for any $s,t\in[1,\infty)$, $\prop(f(s,t))<\infty$ as an operator in $\mathbb{C}[\widetilde E]^G$ and for any $\varepsilon>0$, there exists $S>0$ for any $s>S$ and $t\in[1,\infty)$,
	$\supp(f(s,t))$ lies in the $\varepsilon$-neighborhood of $\widetilde E_{\widetilde Y}\times \widetilde E_{\widetilde Y}$, where $\widetilde E_{\widetilde Y}$ is the restriction of $\widetilde E$ to $Y$.
	
	There is a natural inclusion $i:C^*_{\widetilde B,L,0}(\widetilde E_{\widetilde Y})^G\to (C^*_{\widetilde B,L,0}(\widetilde E)^G)_Y$.
	For any $f\in (C^*_{\widetilde B,L,0}(\widetilde E)^G)_Y$, we have a natural homotopy between $f$ and $f_{s_0}(s,t)=f(s+s_0,t)$, whose support is closed to $\widetilde E_{\widetilde Y}\times \widetilde E_{\widetilde Y}$. This shows that for any $\delta>0$, any $K$-theory element of $(C^*_{\widetilde B,L,0}(\widetilde E)^G)_Y$ admits a representative whose support lies in
	$\widetilde E_{\widetilde Y_\delta}\times \widetilde E_{\widetilde Y_\delta}$, where $Y_\delta$ is the $\delta$-neighborhood of $Y$. As an analogue of \cite[Proposition 3.7]{Yulocalization}, we see that $C^*_{\widetilde B,L,0}(\widetilde E_{\widetilde Y})^G$ and $C^*_{\widetilde B,L,0}(\widetilde E_{\widetilde Y_\delta})^G$ are isomorphic on $K$-theoretical level for small $\delta$. This shows that the $K$-theoretical map $i_*$ is surjective. The injectivity of $i_*$ goes similarly.
	
	Note that $(C^*_{\widetilde B,L,0}(\widetilde E)^G)_{U_1}$ and $(C^*_{\widetilde B,L,0}(\widetilde E)^G)_{U_2}$ are closed ideals of $C^*_{\widetilde B,L,0}(\widetilde E)^G$. And we also have that
	$$(C^*_{\widetilde B,L,0}(\widetilde E)^G)_{U_1}+(C^*_{\widetilde B,L,0}(\widetilde E)^G)_{U_2}=(C^*_{\widetilde B,L,0}(\widetilde E)^G)_{U_1}.$$
	Now the proposition follows from the $K$-theoretical six-term exact sequence (cf: \cite[Lemma 3.1]{HRYMV}).
\end{proof}

In the following Lemma, we show that there exists a natural map
\[
\phi_{L,0}:K_m(C^*_L(\widetilde{B})^H)\otimes K_n(C^*_{L,0}(\widetilde{E},\widetilde{B})^G) \longrightarrow  K_{m+n}(C^*_{\widetilde B,L,0}(\widetilde{E})^G)
\]
compatible with $\phi_0$ defined in Theorem \ref{thm: product maps}.

	\begin{lemma}\label{lemma:factorthu}
		The product map $\phi_0$ defined in Theorem \ref{thm: product maps} factors through the evaluation map
 \[\ev_*: K_{*}(C^*_{\widetilde B,L,0}(\widetilde E)^G) \to  K_{*}(C^*_{L,0}(\widetilde E)^G).\] That is, for any $m,n\in\{0,1\}$ there exists a map
		$$\phi_{L,0}:K_m(C^*_L(\widetilde{B})^H)\otimes K_n(C^*_{L,0}(\widetilde{E},\widetilde{B})^G) \longrightarrow  K_{m+n}(C^*_{\widetilde B,L,0}(\widetilde{E})^G)$$
		such that the following diagram commutes
		$$\xymatrix{K_m(C^*_L(\widetilde{B})^H)\otimes K_n(C^*_{L,0}(\widetilde{E},\widetilde{B})^G)
			\ar[rr]^-{\phi_{L,0}}\ar[ddrr]^{\phi_0} &&K_{m+n}(C^*_{\widetilde B,L,0}(\widetilde{E})^G) \ar[dd]^{\ev_*}\\ &&\\
			&&K_{m+n}(C^*_{L,0}(\widetilde{E})^G)
		}$$
	\end{lemma}
	\begin{proof}
		Without loss of generality, we assume that both $m$ and $n$ are zero. With the same notations as in the proof of Theorem \ref{thm: product maps}, we define $\phi_{L,0}$ by
			$$\phi([f_t]\otimes [g_t])=[(f_{t+s-1}-1)\otimes(g_t-1)+1],$$
		where $(f_{t+s-1}-1)\otimes(g_t-1)+1\in (C^*_{L,0}(\widetilde E)^G)^+$ is given by
		\begin{equation}\label{eq:product loc}
		\begin{split}
		&\Big( \big((f_{t+s-1}-1)\otimes (g_t-1)\big)h\Big) (x,y) \\
		=&\int_{\widetilde B}\int_{\widetilde F} \big(f_{t+s-1}-1\big)(x,x')\otimes
		\big(g_t-1\big)_{x'}(y,y')h(x',y')dy'dx',
		\end{split}
		\end{equation}		
		with $h\in L^2(\widetilde E)$. Here $t\in[1,+\infty)$ is the parameter in $C^*_{L,0}(\widetilde{E})^G$ and $s\in[1,+\infty)$ is the extra parameter induced in
		$C^*_{\widetilde B,L,0}(\widetilde{E})^G$. The expression makes sense as we may assume that the propagation of $f_t$ is small enough. After passing to the matrix algebra and the Grothendieck group, we obtain the map $\phi_{L,0}$. The commuting diagram follows directly from the definition.
	\end{proof}
	\begin{lemma}\label{lemma:localized higher rho}
		With the same notations, for the fiberwise homotopy equivalence $f$, there exists a $K$-theory class $\rho_L(f)\in K_{\dim E}(C^*_{\widetilde B,L,0}(\widetilde{E})^G)$ such that $\ev_*(\rho_L(f))=\rho(f)\in K_{\dim E}(C^*_{L,0}(\widetilde{E})^G)$.
	\end{lemma}
	\begin{proof}
		Denote by $g^B$, $g^{E'}$ and $g^E$ the metric on $B$, $E'$ and $E$ respectively.
		For any $r\in[0,1]$ and $n\in \mathbb N^+$, let $\coprod_n B_{n+r}$ be the disjoint union of countably many $B$'s, where $B_{n+r}$ is equipped with the metric $(n+r)g^B$. Similarly, we define $\coprod_n E'_{n+r}$ and $\coprod_n E_{n+r}$, where $E'_{n+r}$ and $E_{n+r}$ are equipped with the metric $g^{E'}+(n+r)\pi'^*g^B$ and $g^E+(n+r)\pi^*g^B$.
		
		Now we have the following fibrations
		$$\xymatrix{F' \ar[r]^{} &\coprod_n E'_{n+r}\ar[r]^{p'} &\coprod_n B_{n+r}}\text{ and }\xymatrix{F \ar[r]^{} &\coprod_n E_{n+r}\ar[r]^{p} &\coprod_n B_{n+r}},$$
		and the fiberwise homotopy equivalence $\coprod_n f_{n+r}=\coprod f$
		$$\xymatrix{\coprod_n E'_{n+r}\ar[rr]^{\coprod_n f_{n+r}}\ar[dr]_{p'}&& \coprod_n E_{n+r}\ar[dl]^{p} \\ &\coprod_n B_{n+r}&}$$
		Since $f$ preserve the fibration, it induces a family higher rho invariant as in Definition \ref{def:higher rho odd}
		$$\rho(\coprod_n f_{n+r})\in K_{\dim E}(C^*_{L,0}(\coprod_n \widetilde E_{n+r})^G).$$
		Hence for some $\varepsilon>0$ small enough, $\rho(\coprod_n f_{n+r})$ admits a $K_0$(resp. $K_1$)-representative which is a $\varepsilon$-almost projection (resp. unitary) with finite propagation uniformly in $n\in\mathbb N^+$. Along $\widetilde B$ the propagation of such representative restricted to $\widetilde E_{n+r}$ goes to zero as $n$ goes to infinity.
		
		When $r$ varies in $[0,1]$, the above construction gives rise to a path $\rho_L(f)(s)$ for $s\in[1,+\infty)$, along which the propagation along $\widetilde B$ goes to zero as $s$ goes to infinity. Therefore the path defines a class in $K_{\dim E}(C^*_{\widetilde B,L,0}(\widetilde{E})^G)$, which we will also denote by $\rho_L(f)(s)$. Furthermore, $\rho_L(f)(1)$ represents the same $K$-theory class as $\rho(f)$ in $K_{\dim E}(C^*_{L,0}(\widetilde E)^G)$ by definition. This finishes the proof.
	\end{proof}

	Now we are ready to prove Theorem \ref{thm: fiber rho}. We will go through the proof in details only for the case where the dimension of $B$ and $F$ are both even. The other cases are totally similar.
	\begin{proof}[Proof of Theorem \ref{thm: fiber rho}] Let $\rho_L(f)$ be as constructed in the proof of Lemma \ref{lemma:localized higher rho}. We shall show that
	\begin{equation}\label{eq:localized product}
	\phi_{L,0}([D^{sgn}_B] \otimes   \rho(f;B) )= \rho_L(f)\in K_0(C^*_{\widetilde B,L,0}(\widetilde E)^G)
	\end{equation}
	by the Mayer-Vietoris arguments. And the theorem follows from Lemma \ref{lemma:factorthu} and \ref{lemma:localized higher rho}.
	
	We first assume a special case where $E=F\times B$, a trivial fiber bundle over $B$. In this case, the family algebra $C^*_{L,0}(\widetilde E,\widetilde B)^G$ is isomorphic to $C(B)\otimes C^*_{L,0}(\widetilde F)^{\Gamma}$. The product map $\phi_{L,0}$ and the localized higher rho invariant $\rho_L(f)$ are constructed in Lemma \ref{lemma:factorthu} and \ref{lemma:localized higher rho} respectively. Using the same construction as in Section \ref{sec:product formula}, we will obtain line \eqref{eq:localized product} for this trivial case.
	
	Now we turn to the general situation. For simplicity, we assume that the base space $B$ admits a triangulation that makes it a simplicial complex. Assume that the diameter of every simplex on $B$ is small enough so that the restriction of $E$ on every simplex is trivial. Let $B^{(k)}$ be a small open neighborhood of the $k$-skeleton of $B$, which contain the $k$-skeleton of $B$ as a deformation retraction. In particular, $B^{(k)}=B$ when $k$ is $\dim B$. Denote the lift of $B^{(k)}$ to $\widetilde B$ by $\widetilde B^{(k)}$ and the restriction of $\widetilde E$ to $\widetilde B^{(k)}$ by $\widetilde E_{\widetilde B^{(k)}}$.
	
	For any $K$-theory element in $K_*(C^*_{\widetilde B,L,0}(\widetilde E)^G)$, its restriction to $\widetilde E_{\widetilde B^{(k)}}$ is well defined by multiplying the element by the characteristic function of $\widetilde E_{\widetilde B^{(k)}}$ on both side. Similarly for $K_*(C^*_L(\widetilde B)^H)$ and $K_*(C^*_L(\widetilde E,\widetilde B)^G)$. We will prove that line \eqref{eq:localized product} holds when restricted to $\widetilde E_{\widetilde B^{(k)}}$ by induction on $k$.
	
	When $k$ is zero, $B^{(0)}$ is a disjoint union of small balls in $B$, to which the restriction of $E$ is trivial. Therefore line \eqref{eq:localized product} holds on $\widetilde E_{\widetilde B^{(0)}}$. Now we assume that line \eqref{eq:localized product} holds on $\widetilde E_{\widetilde B^{(k)}}$. Let $\Delta$ be the disjoint union of the interior of every $k+1$-simplex in $B^{(k+1)}$. Denote the lift of $\Delta$ to $\widetilde B$ by $\widetilde \Delta$ and the restriction of $\widetilde E$ to $\widetilde \Delta$ by $\widetilde E_{\widetilde \Delta}$. Note that $B^{(k+1)}=\Delta\cup B^{(k)}$. By Proposition \ref{prop:MV}, we have the following six-term exact sequence:
	$$\xymatrix{K_0(C^*_{\widetilde B,L,0}(\widetilde E_{\widetilde B^{(k)}}\bigcap\widetilde E_{\widetilde{\Delta}})^G)\ar[r]&
		{\begin{matrix}
			K_0(C^*_{\widetilde B,L,0}(\widetilde E_{\widetilde B^{(k)}})^G)\\
			\oplus \\
			K_0(C^*_{\widetilde B,L,0}(\widetilde E_{\widetilde{\Delta}})^G)
			\end{matrix}}
		\ar[r]& K_0(C^*_{\widetilde B,L,0}(\widetilde E_{\widetilde B^{(k+1)}})^G)\ar[dd]\\ &&\\
		K_1(C^*_{\widetilde B,L,0}(\widetilde E_{\widetilde B^{(k+1)}})^G)\ar[uu]&
		{\begin{matrix}
			K_1(C^*_{\widetilde B,L,0}(\widetilde E_{\widetilde B^{(k)}})^G)\\
			\oplus \\
			K_1(C^*_{\widetilde B,L,0}(\widetilde E_{\widetilde{\Delta}})^G)
			\end{matrix}}
		\ar[l]&K_1(C^*_{\widetilde B,L,0}(\widetilde E_{\widetilde B^{(k)}}\bigcap\widetilde E_{\widetilde{\Delta}})^G)\ar[l]
	}$$
	
	From the assumption that the diameter of each simplex of $B$ is small, the restriction of $E$ to $\Delta$ or $\Delta\cap B^{(k)}$ is a disjoint union of trivial bundles. Direct computations show that
\[
\partial (\phi_{L,0}([D_B^{sgn}] \otimes   \rho(f;B) )- \rho_L(f))
\]
is trivial in $K_1(C^*_{\widetilde B,L,0}(\widetilde E_{\widetilde B^{(k)}}\bigcap\widetilde E_{\widetilde{\Delta}})^G)$, thus it lies in the image of the map
\[
{\begin{matrix}
			K_0(C^*_{\widetilde B,L,0}(\widetilde E_{\widetilde B^{(k)}})^G)\\
			\oplus \\
			K_0(C^*_{\widetilde B,L,0}(\widetilde E_{\widetilde{\Delta}})^G)
			\end{matrix}}
		\to K_0(C^*_{\widetilde B,L,0}(\widetilde E_{\widetilde B^{(k+1)}})^G).
\]
Then along with the inductive hypothesis, the $K$-theory classes represented by
	$$\phi_{L,0}([D_B^{sgn}] \otimes   \rho(f;B) )- \rho_L(f)$$
	restricted to $\widetilde E_{\widetilde{\Delta}}$, $\widetilde E_{\widetilde B^{(k)}}$ vanishes, which shows that
\[\phi_{L,0}([D_B^{sgn}] \otimes   \rho(f;B) )- \rho_L(f)\]
is the image of trivial class.
 Now line \eqref{eq:localized product} follows when $k=\dim B$.
 \end{proof}

	\section{For special fiber bundle}
	In this section, we show that Theorem \ref{thm: fiber loc} implies the product formula of numerical signature on fibered manifold given by Chern, Hirzebruch and Serre in \cite{CSH57}.

	Consider the fiber bundle $\pi: E\to B $ with fiber $F$, with all those spaces are $4k$-dimensional oriented closed Riemannian manifolds.  Assume that $\pi_1(B)$ acts trivially on $H_{dR}^*(E)$, the de Rham cohomology of $E$. We would like to use our product formula to prove the original formula introduced by Chern, Hirzebruch and Serre in \cite{CSH57}, namely
	\[\sgn(B)\times \sgn(F)=\sgn(E).\]

	Consider the K-theoretic index map
	\[ind_E: K_0(C^*_L(\tilde{E})^{\pi_1(E)})\to K_0(C^*_L(pt))\cong\mathbb Z\]
	induced by the map that crashes the whole space to a point and forgets the group action. Under this the localized index of signature operator will be mapped to its graded Fredholm index, i.e. $\sgn(E)$. Besides we replace $E$ by the base space $B$ and obtain
	$$ind_B : K_0(C^*_L(\tilde{B})^{\pi_1(B)}) \to K_0(C^*_L(pt))\cong \mathbb Z.$$
	
	
	Recall the equivariant family localization algebra $C^*_L(\tilde E,\tilde B)^{\pi_1(E)}$ is the collection of some sections of a C*-bundle over $B$. Any element $s(t)\in C^*_L(\tilde E,\tilde B)^{\pi_1(E)}$ is viewed as a family of operators
	$s(t)_x\in C_L^*(\tilde F)^{\Gamma}$ for $x\in\tilde B$. Thus we define a family index map
	$$ind_{E,B}: K_0(C^*_L(\tilde{E},\tilde{B})^{\pi_1(E)})\to K_0(C(\tilde B))^{\pi_1(B)})\cong K_0(C(B))	$$
	by taking indices along the fiber.
	
	Moreover, we have the following  classical pairing of $K$-homology and $K$-theory
	$$\langle \cdot,\cdot\rangle: K_0(C_L^*(\tilde B)^{\pi_1(B)})\times K_0(C(B))\to\mathbb Z.$$

	From the construction above, we see that the following diagram commutes.
	$$\xymatrix{
		K_0(C_L^*(\tilde B)^{\pi_1(B)})\otimes K_0(C^*_L(\tilde E,\tilde B)^{\pi_1(E)})\ar[rr]^-{\phi}\ar[dd]^{1\otimes ind_{E,B}}&&K_0(C_L^*(\tilde E)^{\pi_1(E)})\ar[dd]^{ind_E}\\
		\\
		K_0(C_L^*(\tilde B)^{\pi_1(B)})\otimes K_0(C(B))\ar[rr]^-{\langle \cdot,\cdot\rangle}&&\mathbb Z
	}$$
	
	Therefore we have the identity
	$$\sgn(E)=\langle [D_B],ind_{E,B}([D_{E,B}])\rangle.$$
	
	Since $F$ is even-dimensional, the family index $ind_{E,B}([D_{E,B}])$ living in $K_0(C(B))$ can be viewed as a virtual vector bundle over $B$. The local picture of such vector bundle is $[ker(D_F)]-[cok(D_F)]$. As we have assumed that $\pi_1(B)$ acts on $H_{dR}^*(E)$ trivially, the virtual bundle is indeed a trivial bundle, i.e. it comes from the inclusion
	$\mathbb Z\cong K_0(C(pt))\to K_0(C(B))$. Moreover, the preimage of $ind_{E,B}([D_{E,B}])$ under the inclusion is actually
	$$\dim ker(D_F)-\dim cok(D_F)=\sgn(F).$$
	Thus the pairing map $\langle\cdot , \cdot \rangle $ is simplified as followed
	\begin{align*}
	\langle [D_B], ind_{E,B}([D_{E,B}])\rangle &=\langle[D_B], \sgn(F)\rangle\\
	&= ind_B([D_B]) \times  \sgn(F)\\
	&= \sgn(B) \times \sgn(F).
	\end{align*}
	From this we obtain the classical product formula of signature of Chern, Hirzebruch and Serre.

	\bibliography{spfbib}

\begin{thebibliography}{10}

\bibitem{CSH57}
S.~S. Chern, F.~Hirzebruch, and J.-P. Serre.
\newblock On the index of a fibered manifold.
\newblock {\em Proc. Amer. Math. Soc.}, 8:587--596, 1957.

\bibitem{HR1}
N.~Higson and J.~Roe.
\newblock Mapping surgery to analysis. {I}. {A}nalytic signatures.
\newblock {\em {$K$}-Theory}, 33:277--299, 2005.

\bibitem{HR2}
N.~Higson and J.~Roe.
\newblock Mapping surgery to analysis. {II}. {G}eometric signatures.
\newblock {\em {$K$}-Theory}, 33:301--324, 2005.

\bibitem{HR3}
N.~Higson and J.~Roe.
\newblock Mapping surgery to analysis. {III}. {E}xact sequences.
\newblock {\em $K$-Theory}, 33(4):325--346, 2005.

\bibitem{HRYMV}
N.~Higson, J.~Roe, and G.~Yu.
\newblock A coarse {M}ayer-{V}ietoris principle.
\newblock {\em Math. Proc. Cambridge Philos. Soc.}, 114(1):85--97, 1993.

\bibitem{HS92}
M.~Hilsum and G.~Skandalis.
\newblock Invariance par homotopie de la signature {\`a} coefficients dans un
  fibr{\'e} presque plat.
\newblock {\em J. Reine Angew. Math.}, 423:73--99, 1992.

\bibitem{JL19}
B.~Jiang and H.~Liu.
\newblock Additivity of higher rho invariant for topological structure group in
  a differential point of view.
\newblock {\em arXiv:1804.09026}, 2019.

\bibitem{PS16}
P.~Piazza and T.~Schick.
\newblock The surgery exact sequence, {K}-theory and the signature operator.
\newblock {\em Ann. K-Theory}, 1(2):109--154, 2016.

\bibitem{Siegel12}
P.~Siegel.
\newblock {\em Homological calculations with the analytic structure group}.
\newblock ProQuest LLC, Ann Arbor, MI, 2012.
\newblock Thesis (Ph.D.)--The Pennsylvania State University.

\bibitem{TWY16}
X.~Tang, R.~Willett, and Y.-J. Yao.
\newblock Roe {$C^*$}-algebra for groupoids and generalized {L}ichnerowicz
  vanishing theorem for foliated manifolds.
\newblock {\em Math. Z.}, 290(3-4):1309--1338, 2018.

\bibitem{Wahl10}
C.~Wahl.
\newblock Product formula for {A}tiyah-{P}atodi-{S}inger index classes and
  higher signatures.
\newblock {\em J. K-Theory}, 6(2):285--337, 2010.

\bibitem{W13}
C.~Wahl.
\newblock Higher {$\rho$}-invariants and the surgery structure set.
\newblock {\em J. Topol.}, 6(1):154--192, 2013.

\bibitem{WXY16}
S.~Weinberger, Z.~Xie, and G.~Yu.
\newblock Additivity of higher rho invariants and nonrigidity of topological
  manifolds.
\newblock {\em To appear in Communications on Pure and Applied Analysis}.

\bibitem{XY14P}
Z.~Xie and G.~Yu.
\newblock Positive scalar curvature, higher rho invariants and localization
  algebras.
\newblock {\em Adv. Math.}, 262:823--866, 2014.

\bibitem{XY19}
Z.~Xie and G.~Yu.
\newblock Higher invariants in noncommutative geometry.
\newblock {\em arXiv:1905.12632, To appear in the special volume dedicated to
  Alain Connes' 70th birthday}, 2019.

\bibitem{Yulocalization}
G.~Yu.
\newblock Localization algebras and the coarse {B}aum-{C}onnes conjecture.
\newblock {\em $K$-Theory}, 11(4):307--318, 1997.

\bibitem{Z16}
R.~Zeidler.
\newblock Positive scalar curvature and product formulas for secondary index
  invariants.
\newblock {\em J. Topol.}, 9(3):687--724, 2016.

\bibitem{Z17}
V.~F. Zenobi.
\newblock Mapping the surgery exact sequence for topological manifolds to
  analysis.
\newblock {\em J. Topol. Anal.}, 9(2):329--361, 2017.

\end{thebibliography}

\end{document}